\documentclass[reqno,12pt]{amsart}
\usepackage{amsmath,amsthm,amssymb,bm,amsrefs}
\usepackage{datetime}
\usepackage[usenames,dvipsnames,table]{xcolor}

\newtheorem{theorem}{Theorem}[section]
\newtheorem{lemma}[theorem]{Lemma}

\newtheorem{conjecture}[theorem]{Conjecture}
\newtheorem{problem}[theorem]{Problem}
\newtheorem{claim}{Claim}

\theoremstyle{remark}
\newtheorem{remark}[theorem]{Remark}

\newcommand{\hF}{\mathcal{F}}
\newcommand{\hG}{\mathcal{G}}
\newcommand{\extr}[2]{\mathcal{F}(#1,#2)}
\newcommand{\mextr}[2]{\mathcal{M}(#1,#2)}
\newcommand{\disj}[1]{\textrm{DG}(#1)} 

\title{Hypergraphs of Bounded Disjointness}

\author{Alex Scott}
\address{Mathematical Institute, 24-29 St Giles', Oxford, OX1 3LB, UK}
\email{scott@maths.ox.ac.uk}
\author{Elizabeth Wilmer}
\address{Department of Mathematics, Oberlin College, Oberlin, OH, 44074, USA}
\email{ewilmer@oberlin.edu}
\thanks{Support for the second author provided by the Great Lakes College Association as part of its New Directions Initiative, made possible by a grant from the Andrew W. Mellon Foundation.}
\date{\today}

\begin{document}

\begin{abstract}
A $k$-uniform hypergraph is $s$-almost intersecting if every edge is disjoint from exactly $s$ other edges.  Gerbner, Lemons, Palmer, Patk\'os and Sz\'ecsi conjectured that for every $k$, and $s>s_0(k)$, every $k$-uniform $s$-almost intersecting hypergraph has at most $(s+1)\binom{2k-2}{k-1}$ edges.  We prove a strengthened version of this conjecture and determine the extremal graphs.  We also give some related results and conjectures.
\end{abstract}

\maketitle

%%%%%%%%%%%%%%%%%%%%%%%%%%%%%%%%%%%%%%%
%                                     %
%                                     %
%      Introduction                   %
%                                     %
%                                     %
%%%%%%%%%%%%%%%%%%%%%%%%%%%%%%%%%%%%%%%

\section{Introduction}\label{sec:intro}

A $k$-uniform hypergraph $\hF$ is \textit{intersecting} if $A\cap B$ is nonempty for all edges $A,B\in\hF$.
Erd\H os, Ko and Rado \cite{EKR} showed that, for $n\ge 2k$, every
$k$-uniform intersecting hypergraph $\hF\subset \binom{[n]}{k}$ has size at most $\binom{n-1}{k-1}$; equality holds for the hypergraph of all $k$-sets containing a fixed element.

It is natural to vary the intersection condition and look at hypergraphs in which some pairs of edges are allowed to be disjoint. A number of authors have addressed the global problem of minimizing the number of disjoint pairs in a hypergraph of given size and order (see Frankl \cite{F77}, Ahlswede \cite{A80}, Ahslwede and Katona \cite{AK78}, Bollob\'as and Leader \cite{BL03}, Das, Gan and Sudakov \cite{DGS}). This paper examines the local version of this question introduced by Gerbner,
Lemons,   
Palmer,
Patk{\'o}s
and Sz{\'e}csi~\cite{GLPPS12}, where each edge is disjoint from a bounded number of other edges. 

Following~\cite{GLPPS12}, we define a hypergraph $\hF$ to be
\textit{$(\le s)$-almost intersecting} if for all $A \in \hF$, there are at most $s$ sets $B \in \hF$ satisfying $A \cap B = \emptyset$, and 
\textit{$s$-almost intersecting} if for all $A \in \hF$, there are exactly $s$ sets $B \in \hF$ satisfying $A \cap B = \emptyset$. 
More generally, let us also say that $\hF$ is \textit{$[a,b]$-almost intersecting} if for all $A \in \hF$
\begin{equation}
a \leq \left|\{B \in \hF : A \cap B = \emptyset\}\right| \leq b. \notag
\end{equation}

The maximum size of a $k$-uniform $(\leq\!\!{s})$-almost intersecting hypergraph was investigated in~\cite{GLPPS12}, where it was shown that  the Erd\H os-Ko-Rado bound continues to hold provided 
$n>n_0(k,s)$.  By contrast, it was also shown in~\cite{GLPPS12} that the maximum size of a $k$-uniform $s$-almost intersecting hypergraph does not grow with the size of the ground set:
every $k$-uniform $s$-almost intersecting hypergraph has 
at most $s\binom{2ks}{ks}$ edges.    Gerbner, Lemons, Palmer, P{\'a}lv{\"o}lgyi, 
Patk{\'o}s and Sz{\'e}csi \cite{GLPPPS13} subsequently improved this bound to 
$(2s-1)\binom{2k}{k}$.

An example of a large $k$-uniform $s$-almost intersecting hypergraph is given by the family
\begin{equation}
\extr{k}{s} = \left\{A \cup \{j\} : A \in \binom{[2k-2]}{k-1},\ j \in \{2k-1,2k,\dots,2k+s-1\} \right\}, \notag
\end{equation} 
which has $(s+1)\binom{2k-2}{k-1}$ edges.
In \cite{GLPPS12}, Gerbner,
Lemons,   
Palmer,
Patk{\'o}s
and Sz{\'e}csi conjecture that for every $k$ and $s>s_0(k)$, this is the maximal size of any $s$-almost intersecting $k$-uniform hypergraph.

We prove this conjecture in Section \ref{sec:simple}.  In fact, we prove a rather stronger result: we show that for every $k\ge2$ there are $R=R(k)$ and $s_0(k)$ such that, for $s>s_0$, every 
$k$-uniform $[R,s]$-almost intersecting hypergraph has at most $(s+1)\binom{2k-2}{k-1}$ edges. 
We also determine all the extremal hypergraphs.  Among the extremal hypergraphs, the family $\extr{k}{s}$ minimizes the number of elements in the base set.

The bound on $R$ that we obtain is rather large, as our argument depends on an application of the Sunflower Lemma of Erd\H os and Rado.  It seems likely that something much smaller would suffice: in fact, we conjecture that for sufficiently large $s$,  $R=1$ is enough.  Note that we cannot take $R=0$, as there are intersecting $k$-uniform hypergraphs of unbounded size (and an intersecting hypergraph is automatically $[0,s]$-almost intersecting).  
However, in Section~\ref{sec:reduced}, we consider the effect of weak disjointness assumptions.
In particular, for the cases $k=2$ and $k=3$, we show that a single pair of disjoint edges suffices to recover the bound $(s+1)\binom{2k-2}{k-1}$ on the number of edges, and that we get the same family of extremal hypergraphs. 
(We remark that \cite{GLPPS12} fully characterizes extremal $s$-almost intersecting graphs for $k=2$ and all~$s$.)

We also prove sharp bounds for
\textit{multihypergraphs}, that is, uniform set systems in which repeated edges are allowed.
As in the hypergraph case, there are $[0,s]$-almost intersecting systems of unbounded size.
In Section~\ref{sec:multi}, we prove that the family
$\mextr{k}{s}$ consisting of $\binom{[2k]}{k}$, with each edge having multiplicity $s$, is the unique extremal examples over $k$-uniform multihypergraphs for the property of being $[1,s]$-almost intersecting.  Note that in the large $k$ and large $s$ limit, $s\binom{2k}{k}$ is about four times as large as $(s+1)\binom{2k-2}{k-1}$.

Finally, in Section~\ref{sec:discuss}, we discuss our results and raise some further questions.

We conclude this section with some definitions. We write $[n]$ for the set $\{1,\dots,n\}$ and $\binom{S}{j}$ for the set of all $j$-element subsets of a set $S$. 
Given a hypergraph $\hF$, its \textit{disjointness graph} $\disj{\hF}$ has vertex set equal to $\hF$, and $A \sim B$ in $\disj{\hF}$ exactly when $A \cap B = \emptyset$. Note that $\hF$ is $[a,b]$-almost intersecting exactly when the minimal and maximal vertex degrees  in $\disj{\hF}$ satisfy $a \leq \delta(\disj{\hF}) \leq \Delta(\disj{\hF}) \leq b$. 
The definitions of \textit{$[a,b]$-almost intersecting} and \textit{disjointness graph} extend directly to \textit{multihypergraphs}, that is, uniform set systems in which repeated edges are allowed. (In the disjointness graph, multiple copies of a single edge correspond to distinct vertices.)

%%%%%%%%%%%%%%%%%%%%%%%%%%%%%%%%%%%
%                                 %
%                                 %
% Multihypergraphs                %
%                                 %
%                                 %
%%%%%%%%%%%%%%%%%%%%%%%%%%%%%%%%%%%

\section{Multihypergraphs}\label{sec:multi}

First, we fully characterize the extremal behavior in the multihypergraph case. Recall that $\mextr{k}{s}$ is the multihypergraph consisting of $\binom{[2k]}{k}$, where each edge occurs with multiplicity $s$. Its disjointness graph consists of $\frac{1}{2}\binom{2k}{k}$ copies of the complete bipartite graph $K_{s,s}$. 

\begin{theorem}\label{thm:multi} For $s \geq 1$, any $k$-uniform $[1,s]$-almost intersecting multihypergraph has at most $s \binom{2k}{k}$ edges.  

The unique multihypergraph achieving this bound is $\mextr{k}{s}$, which is $s$-almost intersecting. 
\end{theorem}

We will use two classical theorems from extremal set theory.  The first is the Bollob\'as theorem on intersections between pairs of sets.

\begin{theorem} \cite{Bollobas65}\label{thm:ci} 
Let $(A_1,B_1),\dots,(A_m,B_m)$ 
be a sequence of pairs of sets with $|A_i|=a$ and $|B_i|=b$ for every  $i$.
If
\begin{enumerate}
\item $A_i \cap B_i = \emptyset$ for $1 \leq i \leq m$, and
\item $A_i \cap B_i \not= \emptyset$ for $1 \leq i , j \leq m$,
\end{enumerate}
then $m \leq \binom{a+b}{b}$.  Furthermore, if $m=\binom{a+b}{a}$ then there is some set $S$ of cardinality $a+b$ such that
the $A_i$ are all subsets of $S$ of size $a$, and $B_i=S\setminus A_i$ for each $i$.
\end{theorem}

We will also need the skew version of this theorem (see Frankl \cite{F82}, Kalai \cite{K84}, Lov\'asz \cite{L77}).

\begin{theorem}\cites{F82,K84,L77}\label{thm:wci} 
Let $(A_1,B_1),\dots,(A_m,B_m)$ 
be a sequence of pairs of sets with $|A_i|=a$ and $|B_i|=b$ for every  $i$.
If
\begin{enumerate}
\item $A_i \cap B_i = \emptyset$ for $1 \leq i \leq m$, and
\item $A_i \cap B_i \not= \emptyset$ for $1 \leq i < j \leq m$,
\end{enumerate}
then $m \leq \binom{a+b}{b}$.
\end{theorem}

Note that the assumptions in Theorem~\ref{thm:wci} are weaker than in Theorem~\ref{thm:ci}; however, there is not a unique extremal graph (for instance, we can take $B_1$ to be empty).

We are now ready to proceed with the proof of Theorem \ref{thm:multi}.

\begin{proof}[Proof of Theorem~\ref{thm:multi}] Let $\hF$ be any such multihypergraph, and let $F=\disj{\hF}$ be its disjointness graph (an edge of $\hF$ with multiplicity $c$ is represented by $c$ distinct vertices in $F$).  We know $\delta(F) \geq 1$ and $\Delta(F) \leq s$. For $A \in \hF=V(F)$, let
$\Gamma(A) = \{B \in \hF : A \cap B = \emptyset\}$ be its neighbourhood in $F$. 

We construct a sequence $(A_1,B_1), (A_2,B_2), \dots$  of pairs of vertices of $F$ according to the following procedure, which we will call the \textit{AB algorithm}: set $i=1$ and $V_1=\hF$. Repeat the following steps until $V_{i}=\emptyset$:
\begin{enumerate}
\item Choose $B_i$ arbitrarily from $V_{i}$. 
\item Since $\delta\geq 1$, we know $\Gamma(B_i) \not= \emptyset$. Let $A_i$ be an arbitrary element of $\Gamma(B_i)$. 
\item Set $V_{i+1} = V_i \setminus \Gamma(A_i)$ and increment $i$.
\end{enumerate}

Let $m$ be the length of the resulting sequence of pairs $(A_i,B_i)$. By the construction we immediately have $A_i \cap B_i = \emptyset$ for $i=1,\dots,m$. Since at stage $i$ we eliminate all sets disjoint from $A_i$ as candidates for any future $B_j$, we know $A_i \cap B_j \not= \emptyset$ for $1 \leq i < j \leq m$. The hypotheses of Theorem~\ref{thm:wci} are satisfied, so $m \leq \binom{2k}{k}$.

Since at the $i$-th step in the AB algorithm we eliminate at most $s$ vertices from $V_{i+1}$, we must have  $ |\hF|/s \leq m$. Thus $|\hF|/s \leq \binom{2k}{k}$, and we have proved the first claim in the theorem.

Now assume that $\hF$ is a $[1,s]$-almost intersecting $k$-uniform multihypergraph with exactly $s\binom{2k}{k}$ edges, and apply the AB algorithm to $F=\disj{\hF}$.  The resulting sequences $A_1, A_2,\dots, A_m$ and $B_1, B_2,\dots, B_m$ have length  at most $\binom{2k}{k}$,  and so the algorithm must eliminate exactly $s$ vertices from $V_{i+1}$ at the $i$th step, for every $i$. Since this must hold for every possible sequence of choices iFpetaln the algorithm, $F$ must be $s$-regular, 
and so $\hF$ itself is $[s,s]$-almost intersecting. 

We claim that for $X,Y \in \hF$, either $\Gamma(X) = \Gamma(Y)$ or $\Gamma(X) \cap \Gamma(Y) = \emptyset$. Why? Assume that the edges $X$,$Y$ are a counterexample, so that there exists $Z \in \Gamma(X)\setminus\Gamma(Y)$. Note that there must then exist a vertex $W \in \Gamma(Y)\setminus\Gamma(X)$, since both $\Gamma(X)$ and $\Gamma(Y)$ contain $s$ elements. Consider running the AB algorithm with $B_1=Z$, $A_1=X$, $B_2=W$, and $A_2=Y$. Then $|V_{2}\setminus V_{3}| = s-|\Gamma(X)\cap\Gamma(Y)| <s$, which is impossible (see Figure~\ref{fig:overlaps}). 

\begin{figure}
\begin{center}
\setlength{\unitlength}{0.06cm}
\begin{picture}(80,40)(0,0)
 \multiput(10,10)(10,0){7}{\circle*{4}}
 \multiput(30,30)(20,0){2}{\circle*{4}}
 \multiput(10,10)(20,0){2}{\line(1,1){20}}
  \multiput(20,10)(20,0){2}{\line(1,2){10}}
   \multiput(30,10)(20,0){2}{\line(0,1){20}}
    \multiput(40,10)(20,0){2}{\line(-1,2){10}}
     \multiput(50,10)(20,0){2}{\line(-1,1){20}}
\put(2,0){\mbox{$Z$}}
\put(70,0){\mbox{$W$}}
\put(28,35){\mbox{$X$}}
\put(48,35){\mbox{$Y$}}
\end{picture}
\end{center}
\caption{\label{fig:overlaps} Here $s=5$, but the non-trivial intersection of $\Gamma(X)$ and $\Gamma(Y)$ forces $|V_2-V_3|=2$ when $B_1=Z$, $A_1=X$, $B_2=W$, and $A_2=Y$. }
\end{figure}
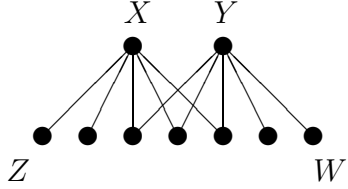

Now let $X$ and $Z$ be two vertices adjacent in $F$. For any $Y \in \Gamma(Z)$, we know $\Gamma(Y) =\Gamma(X)$. That is, every neighbor of $Z$, including $X$, has the same neighbourhood, which must be of size $s$ and contains $Z$. Similarly, every neighbor of $X$ must have the same neighbourhood, which is of size $s$ and contains $X$. We conclude that the connected component containing $X$ and $Z$ in $F$ is isomorphic to the complete bipartite graph $K_{s,s}$, and that $F$ itself consists of $\frac{1}{2} \binom{2k}{k}$ disjoint copies of $K_{s,s}$. 

Let $K= \binom{2k}{k}$, and choose a sequence $F_1G_1,\dots,F_{K/2}G_{K/2}$ of edges of $F$, one from each component.  For each edge $F_iG_i$, we define two pairs $(F_i,G_i)$ and $(G_i,F_i)$: altogether we get $K$ pairs of $k$-sets, and these satisfy the conditions of 
Theorem~\ref{thm:ci}.  It follows that the pairs $\{F_i,G_i\}$ consist of all partitions of some fixed set $S$ of size $2k$ into two sets of size $k$.  Furthermore, replacing any $F_i$ or $G_i$ by a different vertex from (the same vertex class in) the same component must give the same graph.  It follows that each part of each component of $F$ must represent $s$ copies of the same $k$-set.
Thus $\hF$ is actually $\mextr{k}{s}$. 
\end{proof}

%%%%%%%%%%%%%%%%%%%%%%%%%%%%%%%%%%%%%%%%%%
%                                        %
%                                        %
% Simple hypergraphs                     %
%                                        %
%                                        %
%%%%%%%%%%%%%%%%%%%%%%%%%%%%%%%%%%%%%%%%%%

\section{Simple hypergraphs}\label{sec:simple}

For simple hypergraphs, we are able to significantly weaken the assumptions of the conjecture made in \cite{GLPPS12}. Although in Theorem~\ref{thm:simple} we assume only that our hypergraphs are $[R,s]$-almost intersecting for some $R>R_0(k)$, we are able to show that the extremal systems are in fact all $s$-almost intersecting. 

Let us describe the extremal families. Fix disjoint sets $A$, $B$ with $|A|=2k-2$ and $|B|\ge s+1$.  Let $f: \binom{A}{k-1} \to \binom{B}{s+1}$ be any map such that $f(S)=f(A\setminus S)$ for every $S\in \binom{A}{k-1}$.  We then define the $k$-uniform hypergraph ${\mathcal M}_f$ by 
\begin{equation}\label{eq:Mdef}
{\mathcal M}_f = \left\{ S\cup \{x\} : S\in \binom{A}{k-1},\  x\in f(S)\right\}.
\end{equation}
Thus ${\mathcal M}_f$ is the union of $\frac 12 \binom{2k-2}{k-1}$
``double stars'' of the form
$\{ S\cup \{x_1\},\dots, S\cup \{x_{s+1}\}, (A\setminus  S)\cup \{x_1\},
\dots, (A\setminus S)\cup \{x_{s+1}\} \}$.

Each edge of ${\mathcal M}_f$ is disjoint from exactly those edges which have the complementary ``core'' in $A$ and a different ``petal'' in $B$; there are $s$ such edges. Hence all ${\mathcal M}_f$ are, in fact, $[s,s]$-almost intersecting, and all have the same disjointness graph: $\frac{1}{2}\binom{2k-2}{k-1}$ copies of $K_{s+1,s+1}$ minus a matching.

The hypergraphs $\extr{k}{s}$ defined in \cite{GLPPS12} correspond to $|B|=s+1$ and $f(X) = B$ for all $X \in \binom{A}{k-1}$; they clearly minimize the size of the ground set over these families.

\begin{theorem}\label{thm:simple} Fix $k>2$. Then there exist constants $R = R(k)$ and $s_0=s_0(k)$ such that when $s>s_0$, 
any $k$-uniform $[R,s]$-almost intersecting hypergraph has at most $(s+1)\binom{2k-2}{k-1}$ edges. 

The only hypergraphs achieving this bound are those of the form ${\mathcal M}_f$ for some $f$. 
\end{theorem}

\begin{remark} Note that the extremal $[R,s]$-almost intersecting hypergraphs are, in fact, $s$-almost intersecting. 
\end{remark}

\begin{remark} Of course \cite{GLPPS12} covers $k=2$ completely for $s$-almost intersecting hypergraphs. In Section~\ref{sec:reduced} below, we discuss $[1,s]$-almost intersecting hypergraphs in the $k=2,3$ cases. 
\end{remark}

In general, a $\textit{sunflower}$ with $r$ \textit{petals} and \textit{core} $C$ is a collection of sets $Y_1,\dots, Y_r$ such that $Y_i \cap Y_j = C$ for all $i\not=j$.  The disjoint sets $Y_i-C$  are called the \textit{petals}, and they are not allowed to be empty, although the center $C$ can be. (Note that the ``stars'' in $\mathcal{M}_f$ are in fact sunflowers with one-element petals and $(k-1)$-element cores.) The key fact about sunflowers, which we will use in the proof of Theorem~\ref{thm:simple}, is the following classical result of Erd\H{o}s and Rado \cite{ERsunflower}:

\begin{theorem}[Erd\H{o}s-Rado Sunflower Lemma]\label{th:sunflower} Fix $r,k \geq 1$. Any $k$-uniform hypergraph $\hF$ satisfying $|\hF|>k!(r-1)^k$ contains a sunflower with $r$ petals. 
\end{theorem}

The following lemma assures us that extremal examples for Theorem~\ref{thm:simple} avoid a particular kind of pathology. Note that it's simply not true for $k=2$, as the complete bipartite graph $K_{2,s+1}$ contains disjoint edges. 

\begin{lemma}\label{lem:nokdisjoint} Fix $k>2$.  For $s>k^k$, 
no $k$-uniform $[0,s]$-almost intersecting hypergraph with at least $(s+1)\binom{2k-2}{k-1}$ edges contains $k$ mutually disjoint edges.
\end{lemma}
\begin{proof}[Proof of Lemma~\ref{lem:nokdisjoint}] Assume, to the contrary, that $\hF$ is a $k$-uniform $[0,s]$-almost intersecting hypergraph containing mutually disjoint edges $X_1,\dots, X_k$. There are at most $k^k$ edges $Y \in \hF$ such that $Y \cap X_i \not= \emptyset$ for all $i \in [k]$. All other edges in $\hF$ are disjoint from at least one of the $X_i$. Look at degrees in $F=\disj{\hF}$: we must have
\begin{equation}
d_F(X_1)+\dots+d_F(X_k)  \geq (s+1)\binom{2k-2}{k-1}-k^k 
\ge (s+1)(k+1)-k^k, \notag
\end{equation}
since $\binom{2k-2}{k-1}>k+1$ for $k>2$. But then
$d_F(X_1)+\dots+d_F(X_k) > ks$, contradicting
$\Delta(F)\le s$.
\end{proof}

\begin{proof}[Proof of Theorem~\ref{thm:simple}] It will be convenient to introduce a new parameter $r$ and define $R=k^kr^k$. It is then enough to prove that there are functions $r_0(k)$ and $s_0(k,r)$ such that if $r>r_0(k)$ and $s>s_0(k,r)$, then every 
$k$-uniform $(R,s)$-almost intersecting hypergraph has at most $(s+1)\binom{2k-2}{k-1}$ edges. We will find such functions $r_0$ and $s_0$ in the course of the proof. 

Let $\hF$ be a $k$-uniform $[R,s]$-almost intersecting hypergraph with at least $(s+1)\binom{2k-2}{k-1}$ edges.  By repeatedly applying Theorem~\ref{th:sunflower} until too few edges are left to satisfy its hypotheses, we can decompose $\hF$ into a union of $\left\lceil \frac{|\hF|-R}{r}\right\rceil$ sunflowers with $r$ petals each, together with a collection of fewer than $R$ leftover edges. Note that a single core might appear in many sunflowers; however, by Lemma~\ref{lem:nokdisjoint}, none of the sunflowers can have an empty core. Build the $k-1$-uniform \textit{core multihypergraph} $\hG$ by taking the cores of these sunflowers to be edges; if any core has fewer than $k-1$ elements, pad it with new dummy elements (distinct for each edge, so as to introduce no new intersections) to raise the cardinality to $k-1$.  

\begin{claim}\label{cl:talmost}
$\hG$ is $[1,t]$-almost intersecting, where $t = \frac{rs}{(r-k)^2} \geq \frac{s}{r}$.  
\end{claim}

\begin{proof}[Proof of Claim~\ref{cl:talmost}] First consider the upper bound. Suppose a core $C \in \hG$ is disjoint from $T$ other cores, $D_1,\dots, D_T$. Consider a particular $D_i$: 
\begin{itemize}
\item  $D_i$ can intersect at most $k-1$ petals around $C$, since $|D_i|= k-1$ and the petals at $C$ are disjoint;
\item similarly, each edge in the sunflower around $C$ can meet at most $k$ petals around  $D_i$.
\end{itemize}
Thus, the number of disjoint pairs $(X,Y)$, where $X$ is an $\hF$-edge in the sunflower with core $C$, and $Y$ is an $\hF$-edge in the sunflower with core $D_i$ for some $i$, is at least $(r-(k-1))(r-k)$. 
Summing the degrees in $F=\disj{\hF}$ of the $r$ edges in the sunflower with core $C$ gives\begin{equation}
T\cdot(r-k)(r-k+1) \leq rs. \notag
\end{equation}

Now, for the lower bound: any edge  $X$ in $\hF$ is disjoint from at least $R$ other edges of $\hF$, and we omit fewer than $R$ edges total as we construct sunflowers. Hence there must be at least one  edge disjoint from $X$ contributing to a core in $\hG$, and disjointness is preserved by reducing to cores. 
\end{proof}

\begin{claim} \label{cl:sizeofG} Fix $\epsilon>0$. Suppose that $r>6k(4^k)/\epsilon$ and $s>2R/\epsilon$. Then 
\begin{equation}\label{eq:sizeofG}
\left(\binom{2k-2}{k-1}-\epsilon\right) t \leq |\hG| \leq  \binom{2k-2}{k-1} t .
\end{equation}
\end{claim}

\begin{proof}[Proof of Claim~\ref{cl:sizeofG}] The upper bound follows immediately from  Claim~\ref{cl:talmost} and Theorem~\ref{thm:multi}.

For the lower bound, note first that 
\begin{equation}\label{eq:test}
\frac{s}{r} \leq t \leq \frac{s}{r}\left(1+\frac{3k}{r} \right),
\end{equation}
with the second inequality true as long as $r>5k$.  Now, if $|\hG|< \left(\binom{2k-2}{k-1}-\epsilon\right) t$,  then $|\hF| \leq r |\hG| + R$ and~\eqref{eq:test} imply that 
\begin{align}\label{eq:hideous}
|\hF| 
\leq\left( \binom{2k-2}{k-1}-\epsilon\right) \left(1+\frac{3k}{r} \right) s+R
< \binom{2k-2}{k-1} s.
\end{align}
Since $|\hF|\geq\binom{2k-2}{k-1}(s+1)$, this is impossible. 
\end{proof}

\begin{remark} At this point we have obtained an asymptotic version of the main conjecture. By construction we know $r |\hG| \leq |\hF| \leq r |\hG| + R,$ so for any $\epsilon,r,s$ satisfying the conditions of Claim~\ref{cl:sizeofG} we  have  
\begin{equation}\label{asymptoticf}
|\hF|\le(1+\epsilon)(s+1)\binom{2k-2}{k-1}.
\end{equation}
\end{remark}

We now look more closely at the structure of $G=\disj{\hG}$, showing that it must be approximately regular (Claim~\ref{cl:regular}) and has neighorhoods which are either  identical or nearly disjoint (Claim~\ref{cl:nearlydisjoint}). 

\begin{claim}\label{cl:regular} If $\epsilon, r, s$ satisfy the conditions of Claim~\ref{cl:sizeofG}, then 
\begin{equation} (1-\epsilon)t \leq \delta(G) \leq \Delta(G)\leq t.  \notag
\end{equation}
\end{claim}

\begin{proof}[Proof of Claim~\ref{cl:regular}]
If $\delta(G)<(1-\epsilon)t$, then it is possible to run the AB algorithm by taking $B_1$ to be a neighbor of a vertex of minimal degree, $A_1$ to be the vertex of minimal degree itself, and then continuing arbitrarily. We eliminate fewer than $(1-\epsilon)t$ vertices after the first pair, and at most $t$ at each of the following steps. Theorem~\ref{thm:wci} tells us that any run of the AB algorithm must terminate in at most $\binom{2k-2}{k-1}$ steps. Hence 
\begin{align}\label{eq:Gbound}
|\hG| & < (1-\epsilon) t + \left( \binom{2k-2}{k-1}-1\right) t =\left( \binom{2k-2}{k-1}-\epsilon\right) t,
\end{align}
contradicting Claim~\ref{cl:sizeofG}.

\end{proof}

\begin{claim}\label{cl:nearlydisjoint}
If $\epsilon, r, s$ satisfy the conditions of Claim~\ref{cl:sizeofG}, 
then for all $X,Y \in \hF$, either $\Gamma(X)=\Gamma(Y)$ or $|\Gamma(X) \cap \Gamma(Y)| \leq \epsilon t$. 
\end{claim}

\begin{proof}[Proof of Claim~\ref{cl:nearlydisjoint}]
Assume, to the contrary, that there exist $X,Y \in \hG$ with $\Gamma(X)\not=\Gamma(Y)$ and $|\Gamma(X) \cap \Gamma(Y)| >\epsilon t$. By Claim~\ref{cl:regular} we have $(1-\epsilon)t\leq d_G(X),d_G(Y) \leq t$, so there must exist $Z \in \Gamma(X)\setminus \Gamma(Y)$ and $W \in \Gamma(Y)\setminus\Gamma(X)$. Run the AB algorithm with $B_1=Z$, $A_1=X$, $B_2=W$ and $A_2=Y$. After the first pair at most $t$ vertices are eliminated. However, after the second pair, at most $(1-\epsilon)t$ vertices are eliminated, because of the non-trivial intersection of the two neighbourhoods. Again by Theorem~\ref{thm:wci} all vertices must be eliminated after at most $\binom{2k-2}{k-1}$ steps, so we conclude
\begin{equation}\label{eq:Gbound2}
|\hG| \leq t + (1-\epsilon)t + \left(\binom{2k-2}{k-1}-2 \right)t.
\end{equation}
Since the right-hand-sides of \eqref{eq:Gbound2} and~\eqref{eq:Gbound} are equal, Claim~\ref{cl:sizeofG} again gives a contradiction.
\end{proof}

\begin{claim}\label{cl:bipartites} Let $r = 10 k (64^k)$, and suppose that $s>2(160^k)(k^{2k})(64^{k^2})$. Then $G$ is a disjoint union of exactly $\frac{1}{2} \binom{2k-2}{k-1}$ complete bipartite graphs in which the size of each part of each component is between $(0.99)t$ and~$t$. 
\end{claim}

\begin{proof}[Proof of Claim~\ref{cl:bipartites}]
First, set $\epsilon = \frac{1}{16^k}$  and note that $\epsilon, r,s$ then satisfy the conditions for Claims~\ref{cl:sizeofG}, \ref{cl:regular} and~\ref{cl:nearlydisjoint}. Recall that $R=k^kr^k$.

Let $N_1,N_2,\dots$ be the distinct neighbourhoods that occur in $\hG$.  We know that $(1-\epsilon) t \le|N_i|\le t$ for each $i$, and $|N_i\cap N_j|<\epsilon t$ for distinct $i,j$.  So for each $i$, 
$|N_i\setminus \bigcup_{j<i} \Gamma(X_j)| \ge |N_i|-\sum_{j<i}|N_i\cap N_j|\ge (1-i\epsilon)t$.
If there are at least $d=\binom{2k-2}{k-1}+1$ distinct neighbourhoods, then 
\begin{equation}
|\hG|\ge \left|\bigcup_{i=1}^d N_i \right| \geq \sum_{i=1}^d (1-i \epsilon)t =\left(d-\binom{d+1}{2} \epsilon \right)t
>\binom{2k-2}{k-1}t. \notag
\end{equation}
Since $\binom{d+1}{2}<1/\epsilon$, this contradicts Claim~\ref{cl:sizeofG}.

We therefore have that there are at most $\binom{2k-2}{k-1}$ distinct neighbourhoods.  By Claim~\ref{cl:regular}, all neighbourhoods have size at most $t$. Since $\delta(\hG)>0$, the neighbourhoods cover all the vertices. By the lower bound in Claim~\ref{cl:sizeofG}, there are exactly $\binom{2k-2}{k-1}$ neighbourhoods. 

Note that  $\Gamma(x) = \Gamma(y)$, for $x,y$ vertices of $G$, is an equivalence relation on the vertex set of $G$. No equivalence class can contain more than $t$ vertices, since then any vertex in the corresponding neighbourhood would have degree greater than $t$. If any equivalence class contains fewer than $\frac{2t}{3}$ vertices, then the total number of vertices in $G$ is less than  $\left(\binom{2k-2}{k-1}-\frac{1}{3}\right)t$, contradicting Claim~\ref{cl:sizeofG}. Hence every class contains at least $\frac{2t}{3}$ vertices.

Now suppose that $x \in \Gamma_1 \cap \Gamma_2$ witnesses the intersection of two distinct neighbourhoods. Then $x$ has degree at least $\frac{2t}{3}+\frac{2t}{3} = \frac{4t}{3}$, since $x$ is adjacent to every vertex with neighbourhood $\Gamma_1$ and every vertex with neighbourhood $\Gamma_2$. This is impossible, since $\Delta(G) \leq t$. We can conclude that distinct neighbourhoods are in fact fully disjoint. 

For any edge $\{x,y\}$ of $G$, we have $\Gamma(x) = \{z: \Gamma(z)= \Gamma(y)\}$. It follows that the component of $\{x,y\}$ is the complete bipartite graph with parts $\Gamma(x)$ and $\Gamma(y)$, and $G$ therefore has the claimed structure. The lower bound on the size of the classes follows from Claim~\ref{cl:regular}.
\end{proof}

We are now ready to define the promised $r_0(k) = 10k(64)^k$ and $s_0(k)=2(160^k)(k^{2k})(64^{k^2})$, as in the conditions of Claim~\ref{cl:bipartites}. Note that $R=k^kr^k = 10^k k^{2k} 64^{k^2}$. 

The next step of the argument is identical to that at the end of the proof of Theorem~\ref{thm:multi}.  Extract a matching from $G$, taking one edge $(X_i,Y_i)$ from each component for $i=1,\dots, \frac{1}{2} \binom{2k-2}{k-1}$. Recalling that vertices of $G$ are actually edges in $\hG$,  consider the family of pairs of sets
\begin{equation}
(X_1,Y_1), (Y_1,X_1), (X_2,Y_2), (Y_2,X_2),\dots \notag
\end{equation}
By construction, this is a $(k-1)$-uniform cross-intersecting family with $\binom{2k-2}{k-1}$ pairs. By Theorem~\ref{thm:ci}, there exists a set $S$ of size $2k-2$ such that every pair $(X_i,Y_i)$ consists of complementary subsets of $S$, both of size $k-1$. 
As before, replacing any $X_i$ or $Y_i$ with another vertex from the same part of the same component does not change the graph. Hence we know that the core multihypergraph $\hG$ contains all $k-1$ element subsets of a fixed $(2k-2)$-element set $S$, with all multiplicities between $(0.99)t$ and $t$. By the structure of $\hG$, we know that no dummy elements were required in the construction of $\hG$; all the sunflowers had $(k-1)$-element cores and single element petals. 

We once again consider our original simple hypergraph, $\hF$, which we know to be $[R,s]$-almost intersecting and of size at least $(s+1)\binom{2k-2}{k-1}$. 
\begin{claim}\label{cl:blob}
Every edge in $\hF$ intersects $S$ in exactly $k-1$ elements, and every $(k-1)$-element subset of $S$ is contained in exactly $s+1$ edges of~$\hF$. 
\end{claim}
\begin{proof}[Proof of Claim~\ref{cl:blob}] First, let $X$ be a $(k-1)$-element subset of $S$. Then there must be at least $(0.99)s$ edges in $\hF$ containing $X$, since $X$ has multiplicity at least $(0.99)t$ in $\hG$,  each occurrence in $\hG$ corresponds to a sunflower with $r$ petals in $\hF$, and, by Claim~\ref{cl:talmost}, we must have $tr>s$.

Now, if any edge $Z$ of $\hF$ intersects $S$ in $q\leq k-2$ or fewer vertices, it would be disjoint from at least $((0.99)s-k)\binom{2k-2-q}{k-1}$ edges whose cores lie in $S \setminus Z$. This is far more than the allowed $s$ disjointnesses. We conclude that $|Z \cap S| \geq k-1$.

What if $Z \subseteq S$? Then $Z$ intersects every sunflower core; hence, by the preceding paragraph, $Z$ intersects every edge in $\hF$. This is also impossible. 

Note also that if $X$ is contained in \textit{more} than $s+1$ edges of $\hF$, then any edge $W$ of $\hF$ containing $S\setminus X$ is disjoint from at least $s+1$ of the edges containing $X$, a contradiction. 

Since $\hF$ has at least $(s+1)\binom{2k-2}{k-1}$ edges, every element of $\binom{S}{k-1}$ must be contained in exactly $(s+1)$ edges of $\hF$. 
\end{proof}

We have shown that $|\hF| \leq (s+1)\binom{2k-2}{k-1}$. To conclude the main proof, we now assume that~$\hF$ has exactly $(s+1)\binom{2k-2}{k-1}$ edges. 
Let $X,Y \in \binom{S}{k-1}$ be a fixed pair of disjoint sets, and let $x_1, \dots, x_{s+1}$ be the petal vertices over the core $X$. If there exists a vertex $y \not\in \{x_1,\dots,x_{s+1}\}$ such that $Y'=\{y\}\cup Y \in \hF$, then $Y'$ is disjoint from every one of the $s+1$ edges containing $X$, which is impossible. Thus the $s+1$ edges of $\hF$ containing $Y$ are $\{x_1\}\cup Y, \dots, \{x_{s+1}\} \cup Y$. This suffices to show that $\hF$ is indeed of the form $\mathcal{M}_{f}$ for an appropriate function $f$. 
\end{proof}

%%%%%%%%%%%%%%%%%%%%%%%%%%%%%%%%%%
%                                %
%                                %
%  Small values of k             %
%                                %
%                                %
%%%%%%%%%%%%%%%%%%%%%%%%%%%%%%%%%%

\section{Small values of $k$}\label{sec:reduced}

In this section we specialize to the cases $k=2$ and $k=3$, and show that for both values we can in fact take $R=1$ in Theorem 3.1.

\begin{theorem}
\label{thm:simple2} For $s>{\color{Black}13}$, any $[1,s]$-almost intersecting graph has at most  $2s+2$ edges. The only graph achieving this bound is $K_{2,s+1}$.
\end{theorem}

\begin{proof}
Let $\hF$ be any $[1,s]$-almost intersecting graph with $m\ge 2s+2$ edges.  
Let $X$ be the vertex set of $\hF$, and let $F=\disj{\hF}$ be the disjointness graph of $\hF$.  
For $e\in \hF$ and $x\in X$, we shall write $d_F(e)$ for
 the degree in $F$ of $e$ (i.e.~the number of edges disjoint from $e$) and  $d(x)$ for the 
degree in $\hF$ of $x$ (i.e.~the
number of edges that contain $x$).  
Note that $\delta(F)\ge1$: we will carry out most of the proof under the weak assumption that no vertex meets every edge of $\hF$, and only use the assumption that $\delta(F)\ge1$ when we need it.

Next note that:
\begin{itemize}
\item \textcolor{Black}{Any edge $e=\{x,y\}$  of $\hF$  meets (in one vertex) $(m-1)-d_F(e) \geq (m-1)-s \geq s+1$ edges. }
\item \textcolor{Black}{However, $e=\{x,y\}$ also meets $(d(x)-1)+(d(y)-1)$ edges (in one vertex), so $d(x)+d(y) \geq s+3$.}  
\end{itemize}
Now choose a pair $e_1=x_1y_1$ and $e_2=x_2y_2$ of disjoint edges 
(which must exist as $m>3$ and $\hF$ is not a star).   

Suppose that $\min\{d(x_1),d(y_1)\}\ge \textcolor{Black}{8}$.  Let $f$ be any edge incident with $x_1$ that does not meet any of $\{y_1,x_2,y_2\}$.  Then $f$ is disjoint from 
at least $\textcolor{Black}{6}$ of the edges incident with ${\textcolor{Black}y}_1$ 
and at least $\textcolor{Black}{(s+1)}-4=\textcolor{Black}{s-3}$ of the edges incident with $e_2$; since at most $2$ edges are double-counted, $f$ is disjoint from at least $\textcolor{Black}{6+(s-3)-2}>s$ edges, giving a contradiction.  
Arguing symmetrically for $x_2$ and $y_2$ (and relabelling if necessary) we may assume that $d(y_1),d(y_2)\le \textcolor{Black}{7}$ and therefore $d(x_1),d(x_2)\ge \textcolor{Black}{(s+3)-7=s-4}$.

Now if any edge is disjoint from $\{x_1,x_2\}$, it \textcolor{Black}{meets} at most two edges incident with each, and so (as there may be an edge $x_1x_2$) misses at least
$(d(x_1)-2)+(d(x_2)-2)-1=d(x_1)+d(x_2)-5\ge \textcolor{Black}{2s-13}>s$ edges.  This is again a contradiction, so we see that all edges are incident with $x_1$ or $x_2$.

The edge $x_1x_2$ is not present, or it would meet every other edge (this is the only place we use the condition $\delta(F)\ge1$).  Thus any edge incident with $x_1$ meets at most one edge incident with $x_2$; it follows that $x_2$ is incident with at most $s+1$ edges, and similarly $x_2$ is incident with at most $s+1$ edges.  We deduce that $m=2s+2$, and $x_1$ and $x_2$ are each incident with exactly $s+1$ edges.  Furthermore, every edge incident with $x_1$ meets an edge incident with $x_2$.  It follows that $\hF$ is a copy of $K_{2,s+1}$.
\end{proof}

We do not know the smallest possible value for $s$ in Theorem \ref{thm:simple2}; however there are only finitely many graphs with at most \textcolor{Black}{30}
edges and no isolated vertices, and so this could in principle be determined by a finite check. As noted in~\cite{GLPPS12}, there exist $s$-almost intersecting graphs with more than $2s+1$ edges for  $s=1,3,6$. 
  
\begin{theorem}
\label{thm:simple2a} 
 For $s>13$, any $[0,s]$-almost intersecting graph that is not a star has at most 
$2s+3$ edges.   The only graph achieving this bound 
is $K_2+E_{s+1}$.
\end{theorem}

\begin{proof}
Suppose that $\hF$ satisfies the conditions of the theorem and has $m\ge 2s+3$ edges.
We follow the proof of Theorem \ref{thm:simple2} through to the beginning of the final paragraph: at this point we have only used the conditions that $\Delta(F)\le s$ and $\hF$ is not a star.  We now delete the edge $x_1x_2$ if present, and complete the argument, finding that we are left with a copy of $K_{2,s+1}$.  It follows that $x_1x_2$ must have been present, and adding it back gives $K_2+E_{s+1}$, as claimed.
\end{proof}

For $k=3$, we can prove a similar strengthening of Theorem \ref{thm:simple}.

\begin{theorem}
\label{thm:simple3} 
For $s>625$, any $[1,s]$-almost intersecting 3-uniform hypergraph has at most 
$6s+6$ edges. The only hypergraphs achieving this bound 
are of the form $\mathcal{M}_f$ for some function $f$. 
\end{theorem}

\begin{proof}
Let $\hF$ be any $[1,s]$-almost intersecting 3-uniform hypergraph with $m\ge 6s+6$ edges.  
Let $X$ be the vertex set of $\hF$, and let $F=\disj{\hF}$ be the disjointness graph of $\hF$.  
As in the proof of Theorem \ref{thm:simple2}, for $e\in \hF$ and $x\in X$, we shall write $d_F(e)$ for
 the degree  in $F$ of $e$ and  $d(x)$ for the 
degree in $\hF$ of $x$.  
Note that $\delta(F)\ge1$: we will carry out most of the proof under the weak assumption that
for every edge $e$, and vertices $x,y\in e$, there is an edge disjoint from $\{x,y\}$,
and only use the assumption that $\delta(F)\ge1$ when we need it. We note that setting $s>625$ is sufficiently large for all steps below. 

We begin with a useful fact: suppose that some pair of vertices $x$ and $y$ have $t\ge1$ common edges.  There must be some edge $e$ disjoint from $\{x,y\}$; since $e$ is disjoint from all but at most $3$ edges incident with $x$ and $y$, we must have $t\le s+3$.

We now break into three cases, according to the structure of $\hF$. 

\medskip

\noindent {\em Case 1: $\hF$ contains three pairwise disjoint edges, say $e_1,e_2,e_3$.}

There are at most 27 edges meeting all three of these edges, and every other edge is disjoint from at least 1.  It follows that
$d_F(e_1)+d_F(e_2)+d_F(e_3)\ge m-27 \ge 6s-21>3s$, since $s$ is sufficiently large.  This contradicts $\Delta(F)\le s$, so we conclude that $\hF$ does not contain three pairwise disjoint edges.
% s>7 required here
\medskip

\noindent {\em Case 2: There are edges $e$, $f_1$, $f_2$ such that $e$ is disjoint from $f_1$ and $f_2$, and $|f_1\cap f_2|=1$.}

Suppose $f_1\cap f_2=\{y\}$.  
There are at most $s$ edges disjoint from each of $e$, $f_1$, $f_2$, and so at least
$(6s+6)-3s=3s+6$ edges meet all of $e$, $f_1$, $f_2$.  At most 12 edges meet all of 
$e$, $f_1$, $f_2$ and miss $y$, and so at least $(3s+6)-12=3s-6$ edges must meet $e$ and $y$.
At most 3 edges contain $y$ and meet $e$ in two vertices and so at least $3s-9$ edges contain $y$ and meet $e$ in exactly one vertex. 
Let $e=\{x_1,x_2,x_3\}$, and, for $i=1,2,3$, let 
$$E_i=\{f\in\hF: f\cap \{x_1,x_2,x_3,y\} =\{x_i,y\} \}.$$
Since, for each $i$, there are at most $s+3$ edges incident with both $y$ and $x_i$, it follows that 
$|E_i|\ge (3s-9)-2(s+3)=s-15$.

Now any edge that misses both $y$ and $x_i$ must be disjoint from at least $(s-15)-3=s-18$ edges from $E_i$.  It follows that if $f$ misses $y$ then $|f\cap e|\ge 2$, since otherwise $f$ would be disjoint from at least $2(s-18)=2s-36>s$ edges from $E_1\cup E_2\cup E_3$. 
%(provided $s>36$).    
Since there are at most $s$ edges disjoint from $e$, and at most 
$3(s+3)=3s+9$ edges incident with both $e$ and $y$, it follows that there are at least $(6s+6)-s-(3s+9)=2s-3$
edges that meet $e$ and miss $y$.  Thus there are at least $2s-3$ edges that meet $e$ in exactly two vertices.

Finally, consider $f_1$.  There are at most 9 edges that meet $e$ in two vertices and also meet $f_1$. But then at least $(2s-3)-9=2s-12$ edges that meet $e$ in two vertices must miss $f_1$, which gives a contradiction.
% if $s>12$.

\medskip

\noindent {\em Case 3: For every edge $e$, and every pair of edges $f_1$, $f_2$ that are disjoint from $e$, $|f_1\cap f_2|=2$.}

Set $K=11$ and let $G$ be the graph with vertex set $X$ and $xy\in E(G)$ if there are at least $K$ edges from $\hF$ that contain both $x$ and $y$.  For each edge $xy$ of $G$, choose a set $E_{xy}$ of $K$ edges from $\hF$ that contain $x$ and $y$. 
% we actually need K>9, 
% and s depends linearly on K; 
% later  values of s assume K=11, since it does a bit better than K=10.

It will be useful to note a relationship between edges of $G$ and $\hF$.  If $xy$ is an edge of $G$, then every edge that is disjoint from $\{x,y\}$ meets at most 3 edges from $E_{xy}$ and so is disjoint from at least $K-3$ edges from $E_{xy}$; so if there are $t$ edges of $\hF$ disjoint from $\{x,y\}$ then $t(K-3)\le \sum_{e\in E_{xy}}d_F(e) \le sK$ and so $t\le s K/(K-3) = (11/8) s$.

We now consider the structure of $G$:

\begin{itemize}
\item 
Note first that $G$ does not contain three independent edges, or else $\hF$ would contain 3 independent edges (we can pick these greedily). 
% as long as $K>6$

\item 
Next note that, for any edge $x_1y_1$ of $G$, there is at most one edge of $G$ disjoint from $x_1y_1$.  For suppose $x_2y_2$ and $x_2y_3$ are disjoint from $x_1y_1$.  
Since $K=11$, we can greedily extend all three edges to $f_1,f_2,f_3 \in \hF$ such that $f_2 \cap f_3 = \{x_2\}$ and $f_1$ is disjoint from $f_2 \cup f_3$. But we have already ruled this out in Case 2. 
% greedy here requires K at least 6

\item
We next show that $\Delta(G)\le 3$.  Suppose that $xy_i$, $i=1,\dots,4$ are 4 edges of $G$.  Note that every edge of $G$ must contain $x$, since any edge not containing $x$ would be disjoint from at least two edges of form $xy_i$ (which we have just shown does not happen).  Let $f=\{z_1,z_2,z_3\}$ be any edge of $\hF$ that does not contain $x$.  For each $i$, there are at most $s+3$ edges containing both $x$ and $z_i$; and there are at most $s$ edges of $\hF$ disjoint from $f$, so there are at least
$(6s+6)-3(s+3)-s=2s-3$ edges of $\hF$ that do not contain $x$.  There are at most $6(K-1)$ edges of $\hF$ that contain at least two vertices from $\{y_1,y_2,y_3,y_4\}$ (since $G$ does not have any edges among this set), and so there are at least $2s-6K+3$ edges of $\hF$ that miss $x$ and meet  $\{y_1,y_2,y_3,y_4\}$ in at most 1 vertex.  It follows that some 
edge $\{x,y_i\}$ of $G$  is disjoint from at least 
$(3/4)(2s-6K+3)$ edges of $\hF$.  But $(3/4)(2s-6K+3) > sK/(K-3)$,
%( this is where we need K>9. Dependence of s on K is nonlinear; for K=10, get 598, for K=11, get 378.)
giving a contradiction.  We conclude that $\Delta(G)\le 3$.
\end{itemize}
It follows from the facts above that all edges of $G$ are contained in some set $S\subset X$ of size 4. We have not yet shown that $G$ has any edges: if $e(G)=0$, choose $S$ to be any 4-set containing an edge of $\hF$.

We next consider how the edges of $\hF$ intersect $S$.
\begin{itemize}
\item If $S$ contains an edge of $G$ there are at most $s K/(K-3)$ edges of $\hF$ disjoint from $S$; otherwise, 
$S$ contains an edge of $\hF$, so there are at most $s$ edges disjoint from $S$.

\item Suppose that at least $36K$ edges of $\hF$ meet $S$ in exactly one vertex.  Then at least $9K$ of these edges contain the same vertex $x\in S$: let $E'$ be such a set of $9K$ edges.  Now consider the graph $H$ with vertex set $X-S$ and edges $\{e\setminus \{x\}: e\in E'\}$.  $H$ has at least $9K$ edges, and $\Delta(H)\le K-1$ (or we would have an edge of $G$ from $x$ to a vertex outside $S$).  By choosing greedily, we can find a matching of size 5 in $H$; let $f_1,\dots,f_5$ be the corresponding edges of $\hF$.  But now let $e$ be any edge of $\hF$ that does not contain $x$: $e$ meets at most 3 of the $f_i$, so there are two others that are disjoint from $e$ and meet in one vertex, which is a configuration that we have already excluded.  We conclude that at most $36K$ edges of $\hF$ meet $S$ in exactly one vertex.
% all the 36's from this paragraph on need to become 40's if K=10, since 
% the real value for the greedy matching argument is 8K-11. 
\item At most 4 edges of $\hF$ are contained in $S$.
\end{itemize}
It follows that at least $(6s+6)- s K/(K-3) - 36K -4>s+18$ (since $s$ is large enough) 
% here need $s \geq 114$ since $K=11$
edges of $\hF$ meet $S$ in exactly 2 vertices.  If any edge of $\hF$ is disjoint from $S$ then it can meet at most 18 of these edges, giving a contradiction.  We deduce that no edges of $\hF$ are disjoint from $S$, and so at least $(6s+6)-36K-4=6s-36K+2$ edges of $\hF$ meet $S$ in exactly two vertices.  Since no pair of vertices in $S$ belongs to more than $s+3$ edges, it follows that every pair of vertices in $S$ belongs to at least $(6s-36K+2)-5(s+3)=s-36K-13$ edges.  Now if any edge of $\hF$ meets $S$ in only one vertex, it is disjoint from at least $s-36K-16$ edges incident with any pair of vertices in $S$ that it does not meet; there are 3 such pairs, giving a contradiction.
% this is actually the tricky bit: need $s > 618$ with $K=11$; 
% because of the 36->40 issue, it goes up to s > 624 for K=10. 
We conclude that every edge meets $S$ in at least 2 vertices.

Finally, let $S=\{s_1,s_2,s_3,s_4\}$.  If $\hF$ is $[1,s]$-almost intersecting, then no edge of $\hF$ is contained in $S$ (or it would meet every other edge).  No pair $\{s_i,s_j\}$ belongs to more than $s+1$ edges, or any edge containing $S\setminus\{s_i,s_j\}$ would be disjoint from more than $s$ of these edges.  Thus if $|\hF|\ge 6s+6$ then every pair from $S$ is incident with exactly $s+1$ edges, and so every edge incident with $\{s_i,s_j\}$ must meet an edge incident with $S\setminus\{s_i,s_j\}$.  It follows immediately that $\hF$ is an extremal system from the family described in the theorem.
\end{proof}

\begin{theorem}
\label{thm:simple3a} 
Let $s>625$, and suppose that $\hF$ is an $[0,s]$-almost intersecting 3-uniform hypergraph.
If for every edge $e\in \hF$, and vertices $x,y\in e$, there is an edge of $\hF$ disjoint from $\{x,y\}$,
then $\hF$ has at most $6s+10$ edges.   The only hypergraphs achieving this bound 
are given by adding up to 4 edges entirely within the $4$-vertex ``core'' of the extremal hypergraphs of Theorem~\ref{thm:simple3}.
\end{theorem}

\begin{proof}
Follow the proof of Theorem \ref{thm:simple3} through the the final paragraph, and then delete all edges that lie entirely inside $S$.
\end{proof}

%%%%%%%%%%%%%%%%%%%%%%%%%%%%%%%%%%%%%%%
%                                     %
%                                     %
%        Discussion                   %
%                                     %
%                                     %
%%%%%%%%%%%%%%%%%%%%%%%%%%%%%%%%%%%%%%%

\section{Discussion}\label{sec:discuss}

\subsection{}
The results of Section~\ref{sec:reduced} show that for $k=2,3$ we can take $R=1$ in Theorem~\ref{thm:simple}.  We conjecture that a similar result holds for all $k$.

\begin{conjecture}
\label{thm:simpleconj} Fix $k>2$ and let $s>s_1(k)$ be sufficiently large. Then  any $k$-uniform $[1,s]$-almost intersecting hypergraph has at most $(s+1)\binom{2k-2}{k-1}$ edges. 
\end{conjecture}

It seems likely that the only hypergraphs achieving this bound are of the form $\mathcal{M}_f$ for some function $f$, as described in~\eqref{eq:Mdef}. 

\subsection{}
Another way to weaken the conditions of Theorem~\ref{thm:simple} is to drop the condition that every edge needs to be disjoint from some other edge.  Of course, the system can then have unbounded size, as we could take a large star.  However, if we demand only that {\em some} pair of edges is disjoint then, for $k=2,3$, Theorems~\ref{thm:simple2a} and~\ref{thm:simple3a} again determine the hypergraphs of maximal size, which depends only on $k$.

One natural way of expressing the hypothesis of Theorem \ref{thm:simple2a} is to say that no vertex meets all edges.  We might hope for an extension to $k$-uniform hypergraphs for $k>2$, by looking at hypergraphs in which no set of $t$ vertices meets all edges.  In this case, $t=k-2$ is not enough as we can fix a set $S$ of size $2k-3$ and take all edges that meet $S$ in exactly $k-1$ vertices: the resulting hypergraph is intersecting and has unbounded size.  On the other hand, $t=k$ is equivalent to the condition that every edge is disjoint from some other edge, which takes us back to Conjecture~\ref{thm:simpleconj}!  So the only other interesting case is $t=k-1$, for which we conjecture the following.

\begin{problem} 
What is the maximum size of a $k$-uniform,
$[0,s]$-almost intersecting hypergraph $\hF$  in which no set of $k-1$ vertices meets all the edges?  What do the extremal hypergraphs look like?  
\end{problem}

It is easy to see that the size of $\hF$  is bounded (as pointed out to us by Alexey Pokrovskiy \cite{APpers}): 
$\hF$ cannot contain a sunflower with $k+s+1$ edges (as the core of the sunflower would then have to meet every edge),
so the number of edges is bounded by the Sunflower Lemma.

We conjecture that
an extremal example can be obtained by filling in the $(2k-2)$-set at the centre of a hypergraph
of form $\mathcal{M}_f$, giving hypergraphs of size
$$(s+1)\binom{2k-2}{k-1} + \binom{2k-2}{k}.$$
We conjecture that the same bound should hold under the weaker
condition that no edge contains a set of $k-1$ vertices meeting all the edges.
The Sunflower Lemma again gives an upper bound, while Theorems~\ref{thm:simple2a}
and \ref{thm:simple3a} confirm the conjecture in the special cases $k=2,3$.

\subsection{}
Another improvement to Theorem~\ref{thm:simple} would be to bring down the value of $s_0$.  We have not optimized the constants in the proof, but (due to the use of the Sunflower Lemma) the proof gives a bound of order $2^{O(k^2)}$.  It is likely that this is far from the truth.  Can $s_0$ be brought down to a polynomial in $k$?

From below, a few examples show that we cannot hope to bring $s$ all the way down to 1. For $k=2$, \cite{GLPPS12} found several small graphs that are $s$-almost intersecting, but have more than $2s+2$ edges. We note that the complete hypergraph $\binom{[7]}{3}$ is $4$-almost intersecting and has $35 > (4+1)\binom{4}{2}=30$ edges.  Similarly, the complete hypergraph~$\binom{[9]}{4}$ is $5$-almost intersecting with $126 > (5+1)\binom{6}{3}=120$ edges. For all $k\geq 2$, the complete hypergraph $\binom{[2k]}{k}$ is 1-almost intersecting and has more than $2\binom{2k-2}{k-1}$ edges, so we need $s_0>1$. But we cannot even rule out the possibility that we can take $s_0$ to be some constant (independent of $k$).

\subsection{}
Finally, we mention that the problems above have all concerned almost intersecting hypergraphs and multihypergraphs.  Let us define $\hF$ to be \textit{$[a,b]$-almost $t$-intersecting} if for all $A \in \hF$
\begin{equation}
a \leq \left|\{B \in \hF : |A \cap B| < t\}\right| \leq b. \notag
\end{equation}
It would be natural to try to extend the results (and questions) to almost $t$-intersecting hypergraphs for $t\ge2$.

\end{document}